\documentclass[12pt]{article}

\usepackage[top=1in, bottom=1in, left=1in, right=1in]{geometry}
\usepackage{amssymb, amscd, amsmath, amsthm, amsfonts,color}
\usepackage{multicol}
\usepackage{verbatim}
\usepackage{authblk}
\usepackage{graphicx}
\usepackage{ulem}
 
\newcommand{\idiot}[1]{\vspace{5 mm}\par \noindent
\marginpar{\textsc{Note}}
\framebox{\begin{minipage}[c]{.99 \textwidth}
#1 \end{minipage}}\vspace{5 mm}\par}
 
\newcommand{\todone}[1]{\vspace{5 mm}\par \noindent
\marginpar{\textsc{DONE!}}
\framebox{\begin{minipage}[c]{.99 \textwidth}
\tt #1 \end{minipage}}\vspace{5 mm}\par}
 
\renewcommand{\todone}[1]{}
 
\renewcommand{\idiot}[1]{}
 
\newdimen\squaresize \squaresize=12pt
\newdimen\thickness \thickness=0.4pt
 
\def\square#1{\hbox{\vrule width \thickness
   \vbox to \squaresize{\hrule height \thickness\vss
      \hbox to \squaresize{\hss#1\hss}
    \vss\hrule height\thickness}
\unskip\vrule width \thickness}
\kern-\thickness}
 
\def\vsquare#1{\vbox{\square{$#1$}}\kern-\thickness}
\def\blk{\omit\hskip\squaresize}

\def\young#1{
\vbox{\smallskip\offinterlineskip
\halign{&\vsquare{##}\cr #1}}}
 
\def\thisbox#1{\kern-.09ex\fbox{#1}}
\def\downbox#1{\lower1.200em\hbox{#1}}

\newtheorem{thm}{Theorem}[section]
\newtheorem*{thm*}{Theorem}
\newtheorem{prop}[thm]{Proposition}
\newtheorem*{prop*}{Proposition}
\newtheorem{lemma}[thm]{Lemma}
\newtheorem*{lemma*}{Lemma}
\newtheorem{cor}[thm]{Corollary}
\newtheorem*{cor*}{Corollary}
\theoremstyle{definition}
\newtheorem{defn}[thm]{Definition}
\newtheorem*{defn*}{Definition}
\newtheorem{eg}[thm]{Example}
\newtheorem*{eg*}{Example}

\newtheorem{remark}[thm]{Remark}
\newcommand{\K}{\{0,1,\dots,k\}}

\newcommand{\tr}{z}

\def\uu{{\bf u}}

\newcommand{\Z}{\mathbb{Z}}
\newcommand{\R}{\mathbb{R}}
\newcommand{\C}{\mathbb{C}}
 
\newcommand{\Waf}{W_{\textrm{af}}}
\newcommand{\Wfin}{W_{\textrm{fin}}}
\newcommand{\Iaf}{I_{\textrm{af}}}
\newcommand{\Aaf}{A_{\textrm{af}}}
\newcommand{\Qaf}{Q_{\textrm{af}}}
\newcommand{\Paf}{P_{\textrm{af}}}
\newcommand{\Ifin}{I_{\textrm{fin}}}
\newcommand{\Afin}{A_{\textrm{fin}}}
\newcommand{\Qfin}{Q_{\textrm{fin}}}
\newcommand{\Pfin}{P_{\textrm{fin}}}

\newcommand{\Wext}{W_{\textrm{ext}}}
\newcommand{\hstar}{\mathfrak{h}^*}
\newcommand{\h}{\mathfrak{h}}
\newcommand{\haf}{\mathfrak{h}_{\textrm{af}}}
\newcommand{\hstaraf}{\mathfrak{h}_{\textrm{af}}^*}
\newcommand{\hstarfin}{\mathfrak{h}_{\textrm{fin}}^*}
\newcommand{\hfin}{\mathfrak{h}_{\textrm{fin}}}
\newcommand{\Hom}{\textrm{Hom}}
\newcommand{\realrootsaf}{\Phi_{\textrm{re}}}
\newcommand{\realrootsfin}{\Phi_{\textrm{fin}}}
\newcommand{\level}{\textrm{lev}}
\newcommand{\Gr}{\textrm{Gr}_G}

\begin{document}
 
\title{Expansions of $k$-Schur functions in the affine nilCoxeter algebra}

\author[1,2,4]{Chris Berg}
\author[1,4] {Nantel Bergeron}
\author[3] {Steven Pon}
\author[1,4] {Mike Zabrocki}
\affil[1]{Fields Institute\\ Toronto, ON, Canada}
\affil[2]{Universit\'e du Qu\'ebec \`a Montr\'eal, Montr\'eal, QC, Canada}
\affil[3]{University of Connecticut, Storrs, CT, USA}
\affil[4]{York University\\ Toronto, ON, Canada}

\maketitle

\abstract{We  give  a  type  free  formula  for  the  expansion  of  $k$-Schur  functions  indexed  by  fundamental  coweights  within  the  affine  nilCoxeter  algebra.  Explicit  combinatorics  are  developed  in  affine  type $C$.}

\section{Introduction}
 
In \cite{BBTZ}, Berg, Bergeron, Thomas and Zabrocki gave several formulas for the expansion of certain $k$-Schur functions (indexed by fundamental weights) inside the affine nilCoxeter algebra of type A. In particular, they gave an explicit combinatorial description for the reduced words which appear in the expansion.

These coefficients have been studied extensively; they are the coefficients which appear in the product of two $k$-Schur functions. These functions have been identified with representing the homology of the affine Grassmannian in type $A$. 

 They verified their formula by identifying terms in the expansion of a $k$-Schur function with pseudo-translations (elements of the nilCoxeter algebra which act by translating alcoves in prescribed directions). This generalized Proposition 4.5 of Lam \cite{Lam2}, where he gave formulas for $k$-Schur functions indexed by root translations.

Since then, Lam and Shimozono \cite{LS} have discovered a type free analogue of this fact for $k$-Schur functions indexed by coweights. The main goal of this paper is to combine the new result of Lam and Shimozono with the techniques of \cite{BBTZ} to give descriptions of the corresponding reduced words appearing in the decomposition of these $k$-Schur functions, with an emphasis on combinatorics.

Section \ref{type free} develops a type free formula for $k$-Schur functions indexed by special Grassmannian permutations, Section \ref{type c} focuses on the specific combinatorics of affine type $C$, and Section \ref{remaining types} discusses a few examples of the combinatorics in affine types $B$ and $D$.

\subsection{Acknowledgments}

This work is supported in part by CRC and NSERC.
It is the results of a working session at the Algebraic
Combinatorics Seminar at the Fields Institute with the active
participation of C.~Benedetti, Z.~Chen, H.~Heglin, and D.~Mazur. 

This research was facilitated by computer exploration using the open-source
mathematical software \texttt{Sage}~\cite{sage} and its algebraic
combinatorics features developed by the \texttt{Sage-Combinat}
community~\cite{sage-combinat}.

The authors would like to thank Brant Jones for coming to Montr\'eal to explain the combinatorics related to the affine Grassmannian elements in classical types.
 
 The authors would like to thank the reviewer for helpful comments.

\subsection{A brief introduction to root systems}

\subsubsection{Root systems}
 
Let $(I,A)$ be a Cartan datum, i.e., a finite index set $I$ and a generalized Cartan matrix $A= (a_{ij} \mid i,j \in I)$ such that $a_{ii} = 2$ for all $i \in I$, $a_{ij} \in \mathbb{Z}_{\leq 0}$ if $i \neq j$, and $a_{ij} = 0$ if and only if $a_{ji} = 0$.  If the corank of $A$ is 1, then $A$ is of \emph{affine type}; in this case, we write $(\Iaf, \Aaf)$, and let $\Iaf = \{0,1,\ldots, n\}$.  From a Cartan datum of affine type we may recover a corresponding Cartan datum $(\Ifin, \Afin)$ of \emph{finite type} by considering $\Ifin = \Iaf \setminus \{ 0 \}$.  In general, we denote affine root system data with an ``af'' subscript, and finite root system data with a ``fin'' subscript.  Root system data of arbitrary type has no subscript.
 
 Also associated to a Cartan datum we have a root datum, which consists of a free $\Z$-module $\h$, its dual lattice $\hstar = \Hom(\h, \mathbb{Z})$, a pairing $\langle \cdot , \cdot \rangle: \h \times \hstar \to \Z$ given by $\langle \mu,\lambda \rangle = \lambda(\mu)$, and sets of linearly independent elements $\{ \alpha_i \mid i \in I\} \subset \hstar$ and $\{\alpha_i^\vee \mid i \in I\} \subset \h$ such that $\langle \alpha_i^\vee, \alpha_j \rangle = a_{ij}$.  The $\alpha_i$ are known as \emph{simple roots}, and the $\alpha_i^\vee$ are \emph{simple coroots}.  The spaces $\h_\R = \h \otimes \R$ and $\hstar_\R = \hstar \otimes \R$ are the coroot and root spaces, respectively.

 \subsubsection{The affine Weyl group}
 
Associated to a Cartan datum we have the \emph{Weyl group} $W$, with generators $s_i$ for $i \in I$, and relations $s_i^2 = 1$ and  \[ \underbrace{s_is_js_is_j \cdots}_{m(i,j)} = \underbrace{s_j s_i s_j s_i \cdots }_{m(i,j)},\] where $m(i,j) = 2,3,4,6$ or $\infty$ as $a_{ij}a_{ji} = 0,1,2,3$ or $\geq 4$, respectively.  An element of the Weyl group may be expressed as a word in the generators $s_i$; given the relations above, an element of the Weyl group may have multiple \emph{reduced words}, words of minimal length that express that element.  The length of any  reduced word of $w$ is the \emph{length} of $w$, denoted $\ell(w)$.  The \emph{Bruhat order} on Weyl group elements is a partial order where $v < w$ if there is a reduced word for $v$ that is a subword of a reduced word for $w$.  If $v < w$ and $\ell(v) = \ell(w) - 1$, we write $v \lessdot w$.
 
 If $j$ is in $\Iaf$, we denote by $W_j$ the subgroup of $W$ generated by the elements $s_i$ with $i \neq j$.  We denote by $W^j$ a set of minimal length representatives of the cosets $W/W_j$.  The elements of $W^0$ will be referred to as \emph{Grassmannian} elements.
 
\todone{{\bf NB:} It would be better to say more what $\h$ is. I mean as it is stated here, it is too vague. What is the rank of $\h$? Or better just say that $\h^*=\Z d\oplus \bigoplus \alpha^\vee_i$ (this also bring the other point that Steve was mentioning, I also think the root are in $\h$ and coroot in $\h^*$, this need to be cheched). Anyway, we can at least define more precisely right now what is $\h^*$ or $\h$ leaving details for latter\\
{\bf SP:} The idea was to define root spaces in general (rather than for just affine type), in which case this is the only way I can think of to do it (without getting into actual Lie algebras, etc.).  All that's really needed is some vector space and its dual, with these special vectors inside each that have the right relations.  That's why there's confusion between $\h$ and $\hstar$, since some authors like to make coroots in the ``original'' vector space and some like to make them lie in the dual.  It's just a matter of notation, and we just have to be consistent.  I tried to make it everywhere that coroots were in $\h$ and roots in $\hstar$, but I don't know if I caught every instance.  This convention agrees with Kac, and I think with Carter as well (but I don't have Carter with me right now).\\
I guess an alternative way to introduce things would be to just say: Let $A$ be a generalized Cartan matrix of affine type, let $\mathfrak{g}$ be its corresponding Lie algebra, and let $\h$ be the Cartan subalgebra of $\mathfrak{g}$...but I feel less comfortable doing that, since there's a lot hidden in that statement.
}

\todone{{\bf CB:} Is it true that $\langle \alpha^\vee_i, \alpha_j \rangle = a_{ij}$. I thought that there was like a $a_i^{-1}$ in this formula, or am i thinking of $(\alpha_i, \alpha_j)$\\
{\bf SP:} Right.  $(\alpha_i \mid \alpha_j) = a_i^\vee a_i^{-1} a_{ij}$.}

\todone{ I need help here.  I am having a hard time seeing how $\hfin$ is defined.
It just seems to say $\h$ is a free $\Z$-module.
I see that $\haf$ is defined a little later on.  I find it really awkward that
it doesn't seem to have a clear definition beyond it is a `free ${\mathbb Z}$-module'
until we get to equation \eqref{hstarafdef} and \eqref{hafdef} 
and then I can't seem to find $\hfin$
at all.

{\bf SP:} What I was trying to do was to write about root systems in general, and then include notes now and then that restrict to the affine or finite case.  So $\hstar$ is a general root system, but $\hstaraf$ is a root system of affine type, etc.  So this initial section is just to present a combinatorial definition of root systems, to avoid having to get into the Lie algebra side of things.  It's kind of the approach taken by Macdonald or Carter.

So, up until just past equation (4), it's all speaking of arbitrary type root systems.  The next paragraph is on affine type, then a few paragraphs on arbitrary type (where I had accidentally put some af subscripts, that are now removed).  Then I get specific about what $\h$, $\hstar$ look like in affine type, now calling them $\haf$ and $\hstaraf$, and so that's when it makes sense to talk about $\hfin$ for the first time. 

It may not be the best organization -- another way would be to speak entirely about root systems in general, and then at the end put everything about affine root systems.  Or, ignore general root systems completely and speak only of finite and affine type.  In general, if we need to cut, I think this section could be cut down quite a bit by skipping details and just telling results (for example, we don't really need to give the formula for translation actions in general, we could just give formulas for the level 1 and level 0 actions...that way, we don't even have to define the isomorphism between root and coroot spaces, we could just define the actions with a couple formulas).

{\bf CB}: Steve, you know better than any of us what to cut here. Can you cut some of the stuff. I agree, stick to level 0 and level 1. Can we just define translation modulo $c$? We don't actually use the $c$ anyways... I don't see why we need $d$, $\delta$, $c$ or the non degenerate symmetric bilinear form $(,)$. 

Also, if it is ok with everyone I think I am going to go through the whole paper and put a $\diamond$ wherever I see a level 1 action and a $\star$ wherever i see a level 0 action.

{\bf MZ}: Can we move equation (5) and (6) a little higher and say $\hfin$ and $\hstarfin$ are
$\bigoplus_{i \in \Ifin} \Z \alpha_i/\Lambda_i$?

{\bf SP}: I made this change, and I'm in the process of cutting down on extraneous material here.
}

\subsubsection{Weyl group actions}

Given a simple root $\alpha_i$, there is an action $\star$ of $W$ on $\h_\R$ or $\hstar_\R$,
defined by the action of the generators of $W$ as
\begin{align}
s_i \star \lambda &= \lambda - \langle \alpha_i^\vee, \lambda \rangle \alpha_i \quad \textrm{for } i \in I, \lambda \in \hstar_\R\\
s_i \star \mu &= \mu - \langle  \mu, \alpha_i \rangle \alpha_i^\vee \quad \textrm{for } i \in I, \mu \in \h_\R.\label{eq:staraction}
\end{align}
This action by $W$ satisfies $\langle w \star \mu,  w \star \lambda \rangle = \langle \mu,\lambda \rangle$.
 
The set of \emph{real roots} is $\realrootsaf = W \star \{\alpha_i \mid i \in I\}$.  Given a real root $\alpha = w \star \alpha_i$, we have an  \emph{associated coroot} $\alpha^\vee = w \star \alpha_i^\vee$ and an  \emph{associated reflection} $s_\alpha = w s_i w^{-1}$ (these are well-defined, and independent of choice of $w$ and $i$).  

The action by $W$ preserves the \emph{root lattice} $Q = \bigoplus_{i \in I} \Z \alpha_i$ and \emph{coroot lattice} $Q^\vee = \bigoplus_{i \in I} \Z \alpha_i^\vee$.  The \emph{fundamental weights} are 
$\{\Lambda_i \in \hstar_\R \mid \Lambda_i(\alpha_j^\vee) = \delta_{ij} \,  \hbox{ for } i,j \in I\}$, 
and the \emph{fundamental coweights} are $\{\Lambda_i^\vee \in \h_\R \mid \alpha_i(\Lambda_j^\vee) = \delta_{ij}\,   \hbox{ for } i,j \in I \}$.
  These generate the \emph{weight lattice} $P = \bigoplus_{i \in I} \Z \Lambda_i$ and \emph{coweight lattice} $P^\vee = \bigoplus_{i \in I} \Z \Lambda_i^\vee$.  

\todone{{\bf SP}: Are $\hstar$ and $\h$ all mixed up?  I think of it as roots $\alpha$ to be in $\hstar$ and coroots $\alpha^\vee$ to be in $\h$... I need to double check this.

{\bf CB:} This is also how I think of it. weights and roots are dual vectors, like $v \in V_\lambda$ weight space means that $h v = \lambda(h) v$ for all $h \in \h$, so $\lambda \in \hstar$. Right?

 On the paragraph above, a pairing on $\hstar$ should be defined on $(\alpha_i, \alpha_j)$, not $(\alpha_i^\vee, \alpha_j^\vee)$ because $\alpha_i \in \hstar$ and $\alpha_i^\vee \in \h$, right?
 
 {\bf SP:} Right.  It should have been defined on $\haf$.  Fixed.}

We let $\hfin$ denote the linear span of $\{ \alpha_i^\vee  \mid i \neq 0 \}$ and $\hstarfin$ denote the span of $\{ \alpha_i \mid i \neq 0\}$.
Then there is another action $\diamond$ of $W$ on $\hfin \otimes \mathbb{R}$, called the level one action in \cite{Shim}, which is defined by:

\[ s_i \diamond \mu = \left\{
	\begin{array}{ll}
		s_i \star \mu  & \mbox{if } i \neq 0 \\
		s_0 \star \mu - \alpha_0^\vee & \mbox{if } i = 0
	\end{array}
\right. \]
where $\alpha_0^\vee$ is interpreted as $\alpha_0^\vee = - \sum_{i \in \Ifin} \alpha_i^\vee$.

In addition to reflections $s_\alpha$, we have the translation endomorphisms of $\hfin \otimes \R$ given by
\begin{equation}
\label{eq: translations definition}
t_\gamma \diamond \mu  = \mu +  \gamma 
\end{equation}
for $\gamma \in \hfin \otimes \R$.  One can show that $t_\mu t_\gamma = t_{\mu + \gamma}$ and that $t_{w(\mu)} = wt_\mu w^{-1}$ for $w \in \Wfin$, $\gamma, \mu \in \hfin \otimes \R$.

If by abuse of notation we let $\Qfin^\vee = \{t_{\alpha^\vee} \mid \alpha^\vee \in \Qfin^\vee\}$, then the affine Weyl group has an alternate presentation as
$$
\Waf = \Wfin \ltimes \Qfin^\vee.
$$

\begin{remark}
Elements of $\Waf$ corresponding to translations act trivially via the $\star$ action, i.e. $t_\gamma \star \mu = \mu$. 
 \end{remark}

\subsubsection{The extended affine Weyl group} 
 
We can define the \emph{extended affine Weyl group} $\Wext$ by
$$
\Wext = \Wfin \ltimes \Pfin^\vee.
$$
$\Wext$ also has an action on $\hfin \otimes \R$ and $\hstarfin \otimes \R$ via the translation formula \eqref{eq: translations definition}.  Translations in the extended affine Weyl group also act trivially under the $\star$ action on $\hfin \otimes \R$.

\todone{ {\bf CB: }What is $\tau(i)$? We've yet to define an action of $W_{ext}$, as far as I can see. You say that elements of $\Omega$ permute the coroots, so I'm guessing that $\tau(i) = j$ iff $\tau(\alpha_i^\vee) =  \alpha_j^\vee$? Can you be more explicit about the action of $W_{ext}$ and why its restriction to $\Omega$ permutes coroots?

{\bf SP:} The translations by $\alpha$ in \eqref{eq: translations definition} are defined for any $\alpha \in \hstar$, not just roots.  So that formula, together with the usual action of the finite Weyl group, gives an action of the extended affine Weyl group on $\h$ or $\hstar$.

{\bf CB:} I thought you might say that. I think we should change the $\alpha, \beta$ in equation \eqref{eq: translations definition} to $\mu, \nu$ ($\alpha$ usually is a root, or at least an element of the root lattice, so this threw me!)
}

\subsubsection{Affine hyperplanes and alcoves}
  
In $\hfin \otimes \R$, let $H_{\alpha,k} = \{ \mu \mid \langle \mu, \alpha \rangle = k\}$, where $\alpha$ is a finite root and $k \in \Z$.  Reflection over the hyperplane $H_{\alpha,k}$ is equivalent to $t_{k\alpha^\vee}s_{\alpha}$ acting by the $\diamond$ action.  Each hyperplane $H_{\alpha,k}$ is stabilized by the action of $\Waf$ and the set of hyperplanes $\mathcal{H} = \cup_{\alpha,k} H_{\alpha,k}$ is stabilized by the action of $\Wext$.
 
The \emph{fundamental alcove} is the polytope bounded by $H_{\alpha_i,0}$ for $i \in \Ifin$ and $H_{\theta,1}$, where $\theta$ is the \emph{highest root}.  
It is a fundamental domain for the $\diamond$ action of $W$ on $\hfin \otimes \R$.  Therefore, we may identify alcoves with affine Weyl group elements; we define $\mathcal{A}_w$ to be the alcove  $w^{-1} \diamond \mathcal{A}_\emptyset$, where $\mathcal{A}_\emptyset$ is the fundamental alcove.  Additionally, we may identify alcoves with their \emph{centroids}, i.e., the average of the vertices of the alcove.  

\subsubsection{Dynkin diagram automorphisms}
 The length of an element $w \in W$, defined earlier in terms of reduced words, may equivalently be defined to be the number of hyperplanes $H_{\alpha,k}$ that lie between the alcoves $\mathcal{A}_w$ and $\mathcal{A}_\emptyset$.  We can similarly define the length of an element $w \in \Wext$ to be the number of hyperplanes that lie between $\mathcal{A}_w$ and $\mathcal{A}_\emptyset$.  This definition of length implies that there are non-trivial elements of $\Wext$ of length 0.  In fact, it is known \cite{mac03} that
 $$
 \Omega := \{ u \in \Wext \mid \ell(u) = 0 \} \cong \textrm{Aut}(D)  \cong \Pfin^\vee / \Qfin^\vee,
 $$
where $\textrm{Aut}(D)$ is the set of Dynkin diagram automorphisms.

The first of the above isomorphisms can be viewed concretely as follows.  We let $\Omega = \{ u \in \Wext \mid \ell(u) = 0\}$. Let $J = \{ i \in \Ifin \mid \tau(0) = i, \tau \in \textrm{Aut}(D)\}$ be the set of \emph{cominiscule coweights}.  Define $x_\lambda$ to be a minimal length representative of the coset $t_\lambda \Wfin$ for $\lambda \in \Pfin^\vee$.  It can be shown that $ x_\lambda = t_\lambda v_\lambda^{-1}$, where $v_\lambda \in \Wfin$ is shortest element of $\Wfin$ such that $v_\lambda(\lambda) = \lambda_{-}$, and $\lambda_{-}$ is the unique antidominant element of the $\Wfin$-orbit of $\lambda$.  Then $\Omega = \{ x_{\Lambda_i} \mid i \in J\}$, and the element $x_{\Lambda_i}$ corresponds to the Dynkin diagram automorphism sending the node $0$ to the node $i$.  Under this map and the action of $\Wext$ on $\Paf^\vee$ given above, an element $\tau \in \textrm{Aut}(D)$ acts on the coweight lattice $\Paf^\vee$ via $\tau \star \alpha_i^\vee = \alpha_{\tau(i)}^\vee$.  Furthermore,

\todone{{\bf CB:} I highlighted this identity since it comes up later on. Where does it come from? I would like to see a reference } 

 \begin{equation}\label{eqn:Wext} \tau s_i = s_{\tau(i)} \tau \end{equation} for $i \in \Iaf, \tau \in \Omega$.  
  Finally, for $\tau \in \Omega$, $u = s_{i_1} s_{i_2} \cdots s_{i_k} \in W$, we define $\tau(u) = s_{\tau(i_1)}s_{\tau(i_2)} \cdots s_{\tau(i_k)}$.
 
  The extended affine Weyl group can be realized as a semi-direct product of the affine Weyl group and $\Omega$: \[ W_{ext} = W_{af} \ltimes \Omega.\]
 The relation (\ref{eqn:Wext}) describes how elements commute in this realization of $W_{ext}$. 
 
\subsection{k-Schur functions for general type}
  
Let $\mathbb{F} = \mathbb{C}((t))$ and $\mathbb{O} = \mathbb{C}[[t]]$.  The \emph{affine Grassmannian} is defined as $\Gr := G(\mathbb{F})/G(\mathbb{O})$.  $\Gr$ can be decomposed into \emph{Schubert cells} $\Omega_w = \mathcal{B}wG(\mathbb{O}) \subset G(\mathbb{F})/G(\mathbb{O})$, where $\mathcal{B}$ denotes the Iwahori subgroup and $w \in W^0$, the set of Grassmannian elements in the associated affine Weyl group.  The Schubert varieties, denoted $X_w$, are the closures of $\Omega_w$, and we have $\Gr = \sqcup \Omega_w = \cup X_w$, for $w \in W^0$.  The homology $H_*(\Gr)$ and cohomology $H^*(\Gr)$ of the affine Grassmannian have corresponding Schubert bases, $\{ \xi_w\}$ and $\{\xi^w\}$, respectively, also indexed by Grassmannian elements.  It is well-known that $\Gr$ is homotopy-equivalent to the space $\Omega K$ of based loops in $K$ (due to Quillen, see \cite[\textsection 8]{ps86} or \cite{mit88}).  The group structure of $\Omega K$ gives $H_*(\Gr)$ and $H^*(\Gr)$ the structure of dual Hopf algebras over $\mathbb{Z}$.
 
The \emph{nilCoxeter algebra} $\mathbb{A}_0$ may be defined via generators and relations from any Cartan datum, with generators $\uu_i$ for $i \in I$, and relations $\uu_i^2 = 0$ and 
\[ \underbrace{\uu_i\uu_j\uu_i\uu_j \cdots}_{m(i,j)} = \underbrace{\uu_j \uu_i \uu_j \uu_i \cdots }_{m(i,j)},\] where $m(i,j) = 2,3,4,6$ or $\infty$ as $a_{ij}a_{ji} = 0,1,2,3$ or $\geq 4$, respectively.
  
  Since the braid relations are exactly those of the corresponding Weyl group, we may index nilCoxeter elements by elements of the Weyl group, e.g., $\uu(w) = \uu_{i_1} \uu_{i_2}\cdots \uu_{i_k}$, whenever $s_{i_1}s_{i_2}\cdots s_{i_k}$ is a reduced word for $w$.
 
\todone{{\bf CB:}In the last sentence of the paragraph above, it should read ``where $s_{i_1}s_{i_2}\cdots s_{i_k}$ is a reduced word for $w$''. I would have changed it, but I didn't see $s_i$ in this section? Can we just assume the reader understands?
{\bf SP:} Reduced words, etc., are introduced in the Weyl groups and root systems background section.  I fixed the error.}
 
By work of Peterson \cite{Pet}, there is an injective ring homomorphism $j_0: H_*(\Gr) \hookrightarrow \mathbb{A}_0$.  This map is an isomorphism on its image (actually a Hopf algebra isomorphism) $j_0: H_*(\Gr) \to \mathbb{B}$, where $\mathbb{B}$ is known as the \emph{affine Fomin-Stanley subalgebra}.  
 
\todone{{\bf CB:} In the lemma above, shouldn't it end `` for all $\mathbf{w} \in W$''? That makes more sense to me, but i didn't  want to change it without checking with someone.
{\bf SP:} It should be for all $v$.  For a fixed $v$, $w$ ranges over all covers of $v$ in the support of $a$.  But the statement needs to be true for all possible $v$.}
 
\todone{{\bf CB:} Important piece missing here! Before the following definition, we need to mention that in type A, $\mathbb{B}$ is isomorphic to the symmetric function ring with bounded part $k$ and the schubert classes under this isomorphism correspond to $k$-Schur functions.

{\bf SP:} Should this go in the introductory section?  E.g., $k$-Schurs came about in theory of Macdonald functions, then Lam found that they model homology, etc.  Or would it need to be too detailed to go there?

{\bf CB:} I think it would go best here, after everything is defined, although a brief mention of it in the intro is also a good idea.
}
 
\begin{defn}
For $W$ of affine type $X$ and $w\in W^0$ we define the non-commutative $k$-Schur function  $\mathfrak{s}^X_w$  of affine type $X$ to be the image of the Schubert class $\xi_w$ under the isomorphism $j_0$, so $\mathfrak{s}^X_w = j_0 (\xi_w)$. When obvious from context, we will simply write $\mathfrak{s}_w$, omitting the type. This definition comes from the realization of $k$-Schur functions identified with the homology of the affine Grassmannian in \cite{Lam2}. In type $C$ this was first properly developed in \cite{LSS}, and in types $B$ and $D$ this was first developed in \cite{Pon}.
\end{defn}

 \begin{eg}
 In type $A_n^{(1)}$, the elements $\mathfrak{s}^A_w$ are the non-commutative $k$-Schur functions defined in \cite{Lam2}.  One can define a further isomorphism between the affine Fomin-Stanley subalgebra and the ring of symmetric functions generated by the homogeneous symmetric functions $h_\lambda$ with $\lambda_1 \leq n-1$.  Under this isomorphism, the non-commutative $k$-Schur functions are conjectured to
correspond to the $t=1$ specializations of the $k$-Schur functions of 
Lapointe, Lascoux and Morse \cite{LLM} indexed by a $k$-bounded partition corresponding 
to the element $w$ and are isomorphic to the $k$-Schur functions of Lapointe and Morse \cite{LM}.
\end{eg}

\section{A type-free formula}\label{type free}
 
\todone{{\bf NB}: To really have a chance to really check the proof below, I would need
 
1- The definition of $X_\gamma $ in all type
 
2- The lemma 4.7 of \cite{LSS} (I know it is hidden in the notes, can someone put it in Section 2.2 clearly stated)
}
 
 \todone{{\bf SP:} I changed a few $z$'s to $t$'s, so that $t$'s are always true translations, and $z$'s are always pseudo-translations.}
 
Given an element $t = w \tau \in \Wext$ with $w\in \Waf$ and $\tau \in \Omega$, we denote by $\bar{t} = w$ the image of $t$ modulo $\Omega$.  For $\lambda \in \Pfin^\vee$ recall that $t_\lambda \in \Wext$ is the translation which acts on $\haf$ according to \eqref{eq: translations definition}. We let $z_\lambda = \overline{t_\lambda}^{-1}$. In \cite{BBTZ}, the $z_\lambda$ were called \textit{pseudo-translations}.
For a coweight $\gamma$, we let $\Gamma_{\gamma} = \Wfin \gamma$.
\todone{what does the overline represent?}

Independently from Lam and Shimozono \cite{LS}, we have simultaneously discovered a generalization of \cite{BBTZ} and \cite{Lam2} which gives a formula for the $k$-Schur functions indexed by coweight translations. Rather than include our long proof, we will rely on their result.

\begin{prop}[Lam, Shimozono \cite{LS}]\label{mainprop} For a dominant coweight $\gamma$,
 $$\mathfrak{s}_{z_\gamma} = \sum_{\eta \in  \Gamma_\gamma} \uu({z_\eta}).
$$
 \end{prop}
 
Proposition \ref{mainprop} is the starting point for a type-free combinatorial formula generalizing the one that appears in \cite{BBTZ}.
It should be noted though, that this formula does not give reduced words for the terms $z_\eta$; they are defined only as the image of translations in prescribed directions. In Theorem \ref{thm:combinatorial}, we will give a combinatorial description of the explicit reduced words which appear in this sum.

\subsection{Commutation for $k$-Schur functions}

Theorem $5.1$ of \cite{BBTZ} gives a nice commutation relation
for $k$-Schur functions in type $A$ and a generator of the affine nil-Coxeter algebra. 
In this section we will deduce a similar commutation property in Theorem 
\ref{thm: Mike's property} which will allow us to provide more explicit 
formulas for $\mathfrak{s}_w$. 
 
 \begin{defn}\label{gamma notation} We now fix some notation. $\gamma$ will denote the $j^{th}$ fundamental coweight $\Lambda_j^\vee$. If $t_\gamma = z_\gamma^{-1} \tau_\gamma$ then we let $t = t_\gamma, z = z_\gamma$, and $\tau = \tau_\gamma$. 
  \end{defn}
\todone{ {\bf CB:} In the statement of the theorem there is now an abuse of identifying $\Omega$ with dynkin diagram automorphisms. It would be very helpful if this correspondence was written in here.  }

\begin{lemma}\label{lemma:translates}
For a coweight $\gamma$, $z$ is the unique element of $W$ which satisfies $\mathcal{A}_z = \mathcal{A}_\emptyset + \gamma$. 
\end{lemma}

\begin{proof}
The alcove $\mathcal{A}_{z} = z^{-1} \diamond \mathcal{A}_\emptyset$ and $z = \overline{t}^{-1}$, where the action of $t$ corresponds to translation by $\gamma$. The uniqueness follows from the fact that $W$ is in bijection with the set of alcoves.
\end{proof}

\begin{prop}\label{prop:commute}
For $\gamma$, $z$, $\tau$ as in Definition \ref{gamma notation} and $w \in W$, $$ z_{w \star \gamma} w = \tau(w) z.$$
\end{prop}

\begin{proof}

Let $w \in W$. In $\Wext$, we have $w t_{w^{-1} \star \gamma} = t w$. Let $t_{w^{-1}\star\gamma} = z_{w^{-1}\star \gamma}^{-1} \tau'$ for some $\tau' \in \Omega$. Then we have
\begin{align}
w z_{w^{-1}  \star \gamma}^{-1} \tau' &= z^{-1} \tau w,\\
w z^{-1}_{w^{-1}  \star \gamma} \tau' &= z^{-1} \tau(w) \tau,
\end{align}
with the last equality coming from Equation (\ref{eqn:Wext}).
Therefore, we must have $\tau' = \tau$, and \[w z^{-1}_{w^{-1} \star \gamma} = z^{-1} \tau(w).\] Inverting both sides and replacing $w^{-1}$ with $w$ gives the desired result. \end{proof}

The following theorem is a generalization of the commutation property for rectangular $k$-Schur functions found in \cite{BBTZ}.
\begin{thm}
\label{thm: Mike's property} Let $\gamma$ be a fundamental coweight, and let $w \in W$. Then
$$\mathfrak{s}_z \uu(w) = \uu(\tau(w)) \mathfrak{s}_z.$$ 
 \end{thm}

\begin{proof}
This follows from Proposition \ref{prop:commute} and Proposition \ref{mainprop}.
\end{proof}

\todone{{\bf NB:} The proof is right but it is not a clear proof for me. One would like to see explicitly a bijection between the terms in $s_{i}X_\gamma$ and the one in $X_\gamma s_{\tau(i)}$. I will try to work on it when I find time.
}

\subsection{An algebraic formula}
  
We let $W_{0,j}$ denote the subgroup of $W$ generated by the simple reflections $s_i$ with $i \neq 0,j$ and let $W_0^j$ denote the set of  minimal length coset representatives of $W_0 / W_{0,j}$. This subsection provides another formula for the $k$-Schur functions which correspond to fundamental coweights. 
 
 \begin{remark}\label{remark:cosetbijection} Let $\gamma$ denote the $j^{th}$ fundamental coweight, as in Definition \ref{gamma notation}. Then $\Gamma_{\gamma}$ is naturally identified with $W_0^j$. We can construct a bijection between $W_0^j$ and $\Gamma_{\gamma}$ as follows. First we give a map from $W_0$ to $\Gamma_\gamma$: for $v \in W_0$, we define a map $v \rightarrow v(\gamma)$. This map is clearly onto; $\Gamma_\gamma$ is defined to be the image of this map. From equation (\ref{eq:staraction}), we see that $s_i \star \gamma = \gamma$ for $i\neq 0,j$. Therefore, $W_0/W_{0,j}$ is in bijection with $\Gamma_\gamma$. 
  \end{remark}
  
\begin{lemma}\label{lemma:levelaction} Let $w \in W$ and $\mu, \nu \in \haf$. The two actions $\star$ and $\diamond$ are related by:
\[w \diamond (\mu+\nu) = w\diamond \mu + w \star \nu.\] 
\end{lemma}

\begin{proof}
We prove this on the generators $s_i$. If $i$ is not zero, then $s_i$ is linear and the two actions agree, so there is nothing to prove.   If $i=0$, then 
$$s_0 \diamond (\mu + \nu) = s_0 \star (\mu + \nu) - \alpha_0^\vee = (s_0 \star \mu - \alpha_0^\vee) + s_0 \star \nu = s_0 \diamond \mu + s_0 \star \nu.\qedhere$$  \end{proof}

The following proposition is a stepping stone to proving our main theorem; It is used to connect Proposition \ref{mainprop} to Theorem \ref{thm:combinatorial}.
  
 \begin{prop}\label{algebraic_formula} Let $\gamma$ be a fundamental coweight as in Def \ref{gamma notation}. Then
\[\mathfrak{s}_{z} = \sum_{v \in W_0^j} \uu({\tau(v) z v^{-1}}).\]
\end{prop}

\begin{proof}
We will use Proposition \ref{mainprop}; we show that each $\tau(v) z v^{-1}$ is in fact $z_{v \star\gamma}$.

Let $w = \tau(v) z v^{-1}$. We compute \[ \displaystyle \mathcal{A}_{wv} = \mathcal{A}_{\tau(v) z} = \mathcal{A}_{z_{v \star \gamma} v } = v^{-1} z_{v \star \gamma}^{-1} \diamond \mathcal{A}_\emptyset = v^{-1} \diamond (\mathcal{A}_\emptyset + v \star \gamma ) = \mathcal{A}_v + \gamma,\]
where the second equivalence comes from Proposition \ref{prop:commute} and the last two use Lemma \ref{lemma:levelaction}.
Applying $v$ to the left and right of the equation above yields $\mathcal{A}_w = \mathcal{A}_\emptyset + v \star \gamma$. By Lemma \ref{lemma:translates}, $w = z_{v \star \gamma}$. Combined with Remark \ref{remark:cosetbijection}, this concludes the proof.
\end{proof}

\subsection{Towards a general combinatorial formula}

We will outline in this section how to build a combinatorial formula for the $k$-Schur functions indexed by a fundamental coweight. Section \ref{type c} will give more explicit formulas in affine type $C$.

\begin{defn}
A set of combinatorial objects $\mathcal{R}$ will be called a \textit{combinatorial affine Grassmannian set for $W$} if:

\begin{itemize}
\item There is a transitive action of $W$ on $\mathcal{R}$.
\item There exists an element $\emptyset \in \mathcal{R}$  which satisfies $W_0 \emptyset = \{ \emptyset \}$. 
\item The map $W^0 \rightarrow \mathcal{R}$ defined by $w \rightarrow w\emptyset$ is a bijection. 
\end{itemize} 
\end{defn}

Given a combinatorial affine Grassmannian set $\mathcal{R}$, $\mu \in \mathcal{R}$, and the above bijection, we define $w_\mu \in W^0$ by $w_\mu\emptyset = \mu$.

\begin{remark} There is another way of calculating the location of an alcove $\mathcal{A}_w$ given a reduced word of the element $w = s_{i_1} s_{i_2} \cdots s_{i_r}$
that we picture as an \textit{alcove walk}. Given a word $w = s_{i_1} s_{i_2} \cdots s_{i_r}$, the location of $\uu({w})$ 
is calculated by a path starting at $\mathcal{A}_\emptyset$ followed by the alcove $\mathcal{A}_{s_{i_r}},$ then 
\[\mathcal{A}_{s_{i_{r-1}} s_{i_r}},
\mathcal{A}_{s_{i_{r-2}} s_{i_{r-1}} s_{i_r}}, \ldots, \mathcal{A}_{s_{i_{1}} s_{i_{2}}\cdots s_{i_{r-1}} s_{i_r}}.\]
Each of these alcoves is adjacent (see \cite[Proposition 1.1]{BBTZ}) and the word for $w$ determines a
path which travels from the fundamental alcove to $\mathcal{A}_w$ traversing a single hyperplane for each simple reflection in the word. 
\end{remark}

\begin{eg} An example of such a walk which corresponds to the reduced word
$s_2 s_1 s_2 s_1 s_0 s_1 s_0 s_2 s_1 s_0$ appears below.  Each hyperplane is colored according to the simple reflection that corresponds to a crossing of that hyperplane; e.g., crossing a green hyperplane corresponds to an $s_0$, a red hyperplane corresponds to an $s_1$, and a blue hyperplane corresponds to an $s_2$.
\begin{center}
\includegraphics[width=4in]{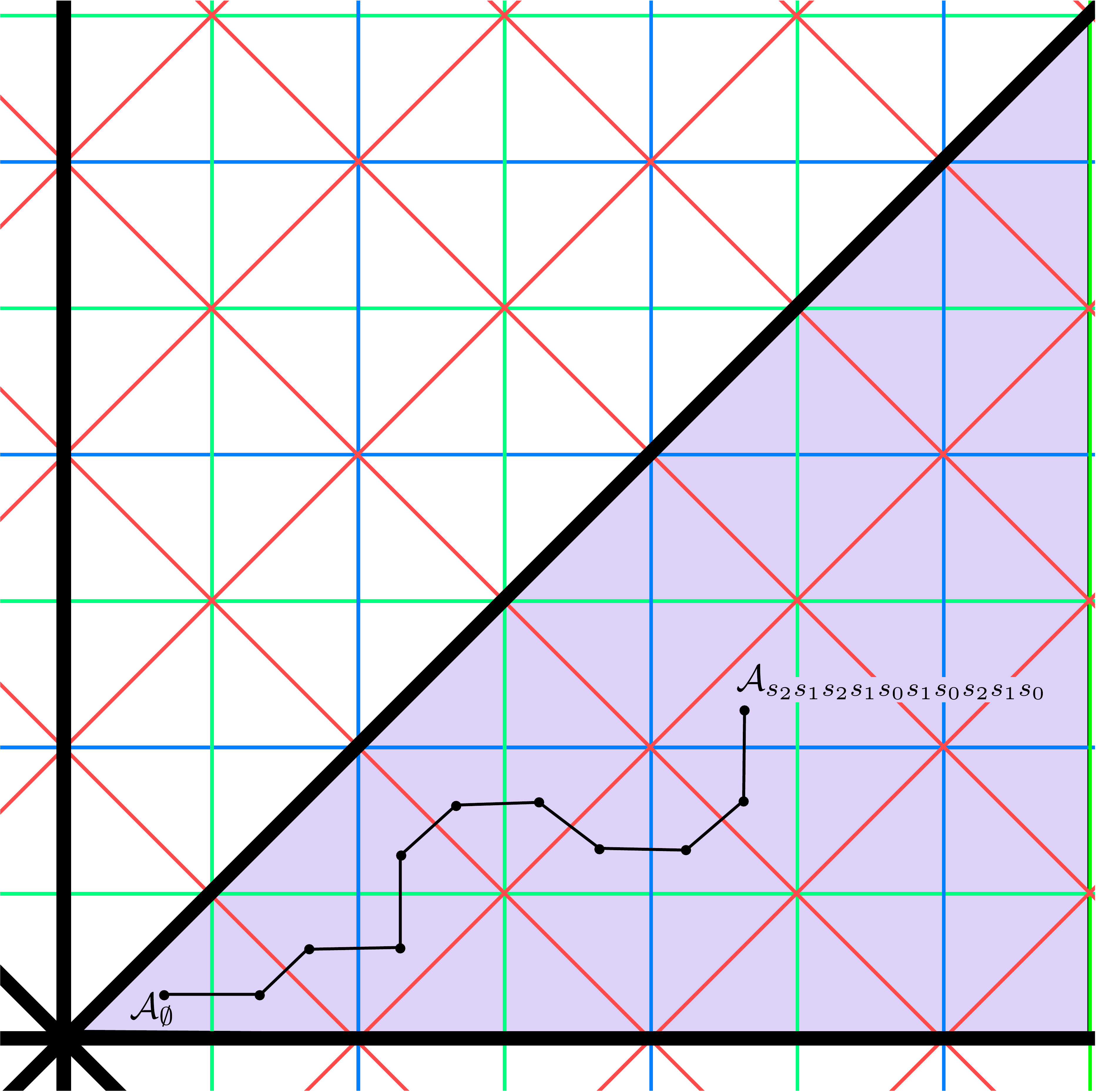}
\end{center}
In the diagram above, the path represents a particular reduced word
for the element of $W^0$ of type $C_2$.  The vertices of this path are in
correspondence with the sequence of alcoves:
$\mathcal{A}_\emptyset \rightarrow 
\mathcal{A}_{s_0} \rightarrow 
\mathcal{A}_{s_1 s_0} \rightarrow 
\mathcal{A}_{s_2 s_1 s_0} \rightarrow 
\mathcal{A}_{s_0 s_2 s_1 s_0} \rightarrow 
\mathcal{A}_{s_1 s_0 s_2 s_1 s_0} \rightarrow 
\mathcal{A}_{s_0 s_1 s_0 s_2 s_1 s_0} \rightarrow 
\mathcal{A}_{s_1 s_0 s_1 s_0 s_2 s_1 s_0} \rightarrow 
\mathcal{A}_{s_2 s_1 s_0 s_1 s_0 s_2 s_1 s_0} \rightarrow 
\mathcal{A}_{s_1 s_2 s_1 s_0 s_1 s_0 s_2 s_1 s_0} \rightarrow 
\mathcal{A}_{s_2 s_1 s_2 s_1 s_0 s_1 s_0 s_2 s_1 s_0}$.
\end{eg}

\todone{{\bf CB:} Mike, want to draw us a type 'C' alcove walk picture? =)
{\bf MZ}: 2 or 3 dimensions? :)  I'll work on it. }

We can define $x \in \hfin \otimes \R$ to be on the \emph{positive} or \emph{negative} side of the hyperplane $H_j := H_{\alpha_j,0}$ by $\langle x, \alpha_j \rangle > 0$ or $\langle x, \alpha_j \rangle < 0$, respectively.

\begin{lemma}\label{lemma:crosstwice} (see for instance \cite{W})
Minimal length expressions of $w \in W$ correspond to alcove walks which do 
not cross the same affine hyperplane twice.
\end{lemma}

\begin{lemma}\label{lemma:jgrass}[See for instance \cite{carter}]
Let $j \in \{1,2, \dots, k \}$. Then $w$ has a right $j$ descent ($w s_j < w$) if and only if the alcove $\mathcal{A}_w$ is on the negative side of the hyperplane $H_j$. 
\end{lemma}

\begin{lemma}\label{lemma:injective}
For all $v \in W_0^j$, $\tau(v) z \in W^0$. 
 \end{lemma}

\begin{proof}
By Proposition \ref{prop:commute}, $\tau(v) z = z_{v \star \gamma} v$. Therefore, the alcove \[ \mathcal{A}_{\tau(v) z} = \mathcal{A}_{z_{v \star \gamma} v} = (z_{v \star \gamma}v)^{-1} \diamond \mathcal{A}_\emptyset = \] \[ v^{-1} z^{-1}_{v \star \gamma} \diamond \mathcal{A}_\emptyset =  v^{-1} \diamond (\mathcal{A}_\emptyset + v \star \gamma) = \mathcal{A}_{v} + \gamma,\] by Lemma \ref{lemma:levelaction}. Since $v \in W_0^j \subset W^j$, the only right descent of $v$ is a $j$ descent, so for $x \in \mathcal{A}_v$ and $i \neq 0,j$ we have $\langle x, \alpha_i \rangle \geq 0$, by Lemma \ref{lemma:jgrass}. Furthermore, $v \in W_0^j \subset W_0$, so $\langle x, \alpha_j \rangle \geq -1$ for $x \in \mathcal{A}_v$ 
(as every alcove corresponding to $v \in W_0$ has a vertex at the origin). Combining these two facts, we get that $\langle x + \gamma, \alpha_i \rangle \geq 0$ for all $i \neq 0$ (since $\langle \gamma, \alpha_j \rangle \geq 1$). Therefore, the alcove $\mathcal{A}_v + \gamma$ is dominant, so the corresponding element is Grassmannian, i.e. $\tau(v) z \in W^0$. 
\end{proof} 

 We let $w_0^j$ be the (unique) maximal length element of $W_0^j$. The set $\mathcal{R}$ inherits a partial order from $W^0$; for $\mu, \nu \in \mathcal{R}$ we say $\mu \leq \nu$ whenever $w_\mu \leq w_\nu$.
For $\mu, \nu \in \mathcal{R}$ with $\nu \leq \mu$, we let $w_{\mu / \nu}:=w_\mu w^{-1}_\nu$.
  
 \begin{thm}\label{thm:combinatorial} Let $R = z \emptyset$ and $S = \tau(w_0^j) z\emptyset$. Then
 $$ \displaystyle \mathfrak{s}_{z} = \sum_{S \leq \lambda \leq R} \uu({w_\lambda \tau^{-1}(w_{R/\lambda})}).$$
 \end{thm}

\begin{proof}
We construct a map $\Phi: W_0^j \rightarrow \mathcal{R}$ by sending $v \in W_0^j$ to  $\Phi(v) = \tau(v) z \emptyset$. By Lemma \ref{lemma:injective}, $\Phi$ is injective and hence $\Phi$ is a bijection on its image, which is precisely all $\lambda \in \mathcal{R}$  satisfying $S \leq \lambda \leq R$. In other words, $w_\lambda = \tau(v) z$ whenever $\Phi(v) = \lambda$. 

Now $w_{R/\lambda} = w_R w_\lambda^{-1}$, so $w_{R/\lambda} = z (\tau(v) z)^{-1} = \tau(v^{-1})$. 
Therefore $w_\lambda \tau^{-1} (w_{R/\lambda}) = \tau(v) z \tau^{-1}(\tau(v^{-1})) = \tau(v) z v^{-1}$.
By Proposition \ref{algebraic_formula}, the theorem follows.
\end{proof}

\begin{remark}
It should be noted that Theorem \ref{thm:combinatorial} gives the reduced words which appear in the expansion of $\mathfrak{s}_z$; they are precisely the reduced words which correspond to objects from $\mathcal{R}$. Once the bijection between $W^0$ and $\mathcal{R}$ is understood, the terms in Theorem \ref{thm:combinatorial} are as well.

  In this sense the theorem is stronger than Proposition \ref{mainprop}, although its proof relies entirely on the proposition. In particular, this theorem generalizes Definition 2.1 of \cite{BBTZ} to all affine types.
\end{remark}
 
\section{Type C combinatorics}\label{type c}
 
 As an application of Theorem \ref{thm:combinatorial}, we use this section to develop the combinatorics of affine type $C$. 
 
\subsection{Type C root system background}
Fix an integer $k > 1$. We recall some facts about roots and weights in affine type $C$ (see \cite{carter} for more details). We let $\epsilon_1, \dots, \epsilon_k$ denote an orthonormal basis for $V:=\mathbb{R}^k \equiv \hfin \otimes \mathbb{R}$.
We realize $\alpha_1 = \epsilon_1 - \epsilon_2, \alpha_2 = \epsilon_2 - \epsilon_3, \dots,\alpha_{k-1} = \epsilon_{k-1} - \epsilon_k, \alpha_k = 2 \epsilon_k$ as the simple roots of finite type $C_k$. 

The fundamental weights are realized as $\Lambda_i := \epsilon_1 + \dots + \epsilon_i$ for $i = 1, \dots, k$. The fundamental coweights are $\Lambda_i^\vee = 2 \Lambda_i$ for $i \neq k$ and $\Lambda_k^\vee = \Lambda_k$. 

The fundamental coweights $\Lambda_1^\vee, \dots, \Lambda_{k-1}^\vee$ also belong to the coroot lattice $\Qfin^\vee $. The elements $t_{\Lambda_i^\vee}$ actually equal $z^{-1}_{\Lambda_i^\vee}$ (for $i \neq k$) in $\Wext$, i.e. these elements have trivial Dynkin diagram automorphisms (as compared to type $A$, where all fundamental coweights correspond to distinct non-trivial Dynkin diagram automorphisms). 

Since $\Lambda^\vee_k$ is not in $\Qfin^\vee$, $t_{\Lambda_k^\vee}$ corresponds to a non-trivial Dynkin diagram automorphism. In affine type $C$ there is only one such automorphism, which we will denote $\tau$. It is defined by $\tau(i) = k-i$ for all $i \in \{0,1,\dots,k \}$.

We let $W$ denote the affine Coxeter group of type $C$. Recall it is generated by $s_0, s_1, \ldots, s_k$ subject to the relations:
   \begin{align*}
         s_i^2 = 1  & \textrm{ for } i \in \K,\\
          s_is_j = s_js_i   & \textrm{ if } i-j \neq \pm 1,\\
           s_is_{i+1}s_i = s_{i+1}s_is_{i+1}  & \textrm{ for } i \in \{1, \dots, k-2\},\\
                    s_i s_{i+1} s_i s_{i+1} = s_{i+1} s_i s_{i+1} s_i & \textrm{ for } i \in \{0, k-1\}.
\end{align*}

\subsection{Bijection between Grassmannian elements and symmetric $2k$-cores}
 
\begin{defn}
The \emph{hook length} of a cell $x$ in the Young diagram of a partition $\lambda$ is the number of cells of the Young diagram of $\lambda$ to the right of $x$ and above $x$, including the box $x$. A partition $\lambda$ is called an $n$-core if for every cell $x$ in the Young diagram of $\lambda$, $n$ does not divide the hook length of $x$.
\end{defn}
 
In \cite{HJ}, Hanusa and Jones give a construction for a combinatorial affine Grassmannian set for $W$ for all classical affine $W$ (the affine Grassmannians corresponding to $B_{k}^{(1)}/B_k, C_{k}^{(1)}/C_k, D_{k}^{(1)}/D_k, B_{k}^{(1)}/D_k$).

In affine type $C$, the set $\mathcal{R}$ of combinatorial affine Grassmannian elements they give are symmetric $2k$-core partitions (symmetry is with respect to transposing the partition). We give a short outline of the action of $W$ on $\mathcal{R}$ as follows:
 
Let the residue of a cell $(i,j)$ of a Young diagram be:
 
$res(i,j) = \left\{ \begin{array}{ll} j-i \mod\, 2k & \textrm{ if }0 \leq (j-i)  \mod \,2k \leq k  \\ 2k - ((j-i)  \mod\, 2k) & \textrm{ if } k< (j-i)  \mod\, 2k < 2k
\end{array} \right.$

We can then define an action on symmetric 2k-core partitions by letting $s_i \lambda =$  
\[\left\{ \begin{array}{ll} \lambda \cup \{\textrm{residue } i \textrm{ cells}\} & \textrm{ if } \lambda \textrm{ has addable cell of residue } i \\
\lambda \setminus \{ \textrm{residue } i \textrm{ cells}\} & \textrm{ if } \lambda \textrm{ has  removable cell of residue } i \\ \lambda & \textrm{ else } \end{array} \right.\]
 
\begin{thm}[Hanusa, Jones \cite{HJ}]
 The action of $W$ on $\mathcal{R}$ described above makes $\mathcal{R}$ into a combinatorial affine Grassmannian set for $W$. 
 \end{thm}
 
\begin{eg}
Let $k= 3$ and let $w = s_{1} s_{2}s_{3}s_{2}s_{0}s_{1}s_{0} \in W^0$. Then $w$ corresponds to the symmetric $6$-core $(6,3,2,1,1,1)$.
$$\young{1\cr2\cr3\cr2&1\cr1&0&1\cr0&1&2&3&2&1\cr}$$
\end{eg}

\begin{remark} Symmetric $2k$-core partitions have extraneous data. Half of the partition is determined from the other, so we will sometimes think of a symmetric $2k$-core as a diagram with boxes $(i,j)$ for $j \geq i$. We call such a diagram a \emph{shifted diagram}.
\end{remark}

\begin{eg}
Let $k=3$ and $w = s_{1} s_{2}s_{3}s_{2}s_{0}s_{1}s_{0}$ as above. Then the shifted diagram for the $6$-core is:

$$\young{\blk&0&1\cr0&1&2&3&2&1\cr}$$

\end{eg}

\begin{eg} A portion of the lattice of symmetric $4$-cores coming from the action
of the affine Coxeter group of type $C_2^{(1)}$ is pictured below.

\begin{center}
\includegraphics[width=4in]{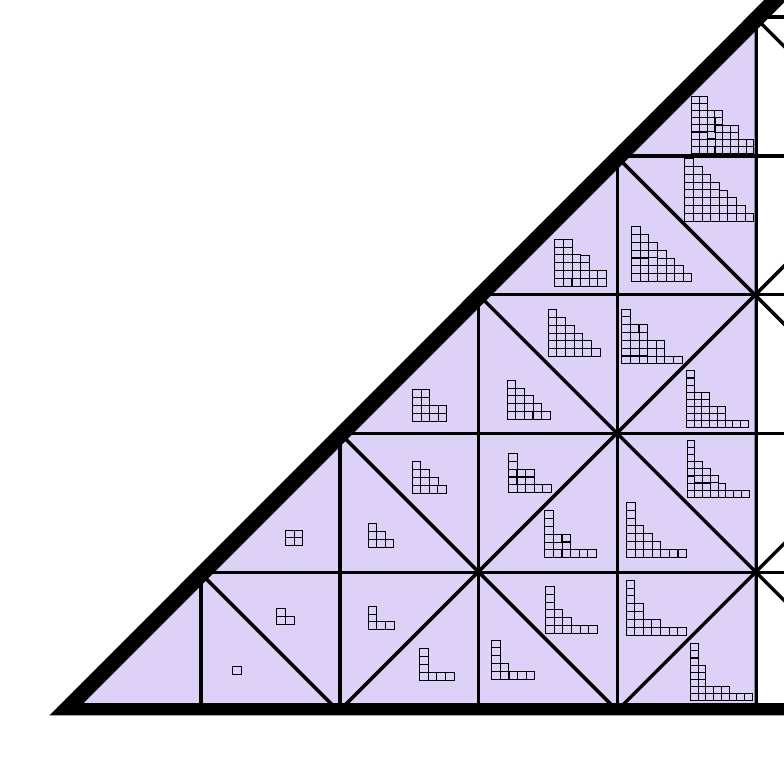}
\end{center}

The pseudo-translation $z_{\Lambda_1^\vee}$ corresponding to
the fundamental coweight $\Lambda_1^\vee= 2\epsilon_1$ takes the fundamental
alcove to the alcove indexed by the symmetric $4$-core $(4,1,1,1)$
and the pseudo-translation $z_{\Lambda_2^\vee}$ corresponding
to the fundamental coweight $\Lambda_2^\vee = \epsilon_1+\epsilon_2$ 
takes the fundamental alcove to the alcove indexed
by the symmetric $4$-core $(2,2)$.
\end{eg}

\subsection{The words and cores corresponding to fundamental coweights}
  
Each $s_i$ acts on $V$ by reflecting across the hyperplane corresponding to the simple root $\alpha_i$ for $i \neq 0$ and reflecting across the affine hyperplane $H_{\theta,1} = \{v\in V : \langle v, \theta \rangle = 1 \}$, where $\theta$ is the highest root, for $i=0$. Specifically, if we let $(a_1, \dots, a_k) \in V$ represent $\sum_i a_i \epsilon_i$, then:\\
 
$s_i  \diamond (a_1, \dots, a_k) =  \left\{  \begin{array}{ll} (a_1, \dots, a_{i+1}, a_i, \dots, a_k) & \textrm{ for } i =1,\dots, k-1;\\
 (a_1, \dots, a_{k-1}, -a_k)& \textrm{ for } i = k;\\
 (2-a_1, \dots, a_k) & \textrm{ for } i = 0. \end{array} \right.$\\

For $i \leq k+1$ we let $w_i := s_{i-1} s_{i-2} \cdots s_1 s_0 \in W$.
\begin{lemma}\label{lemma:w_i}
For $i \leq k$, the element $w_i$ acts on $v = (a_1, \dots, a_k) \in V$ by:
\[ w_i \diamond v = (a_2, a_3, \dots, a_i, 2-a_1, a_{i+1}, \dots, a_k ).\] 

Also, \[w_{k+1} \diamond v = (a_2, a_3, \dots, a_k, a_1 - 2) \]

\end{lemma}
 
\begin{proof}
Simple calculation using Weyl group action described above.
\end{proof}
  
\begin{lemma}\label{lemma:www}
$w_{k+1}^{-1} w_k w_j^{-1}  \diamond (a_1, \dots, a_k) = (a_j-2, a_1, a_2, \dots, \widehat{a_j}, \dots, a_k)$.
\end{lemma}

\begin{proof}
Simple calculation using Lemma \ref{lemma:w_i}.
 \end{proof}
 
If $G_\emptyset$ is the centroid of $\mathcal{A}_\emptyset$, then \[G_\emptyset = \frac{1}{k+1} \sum_i \Lambda_i = (\frac{k}{k+1}, \frac{k-1}{k+1}, \dots, \frac{1}{k+1}).\]

Recall that for a fixed $j$ we let $\gamma$ denote the coweight $\Lambda_j^\vee$.

\begin{lemma}
For $j\neq k$, $z_\gamma =  (w_j w_k^{-1} w_{k+1})^j$.
\end{lemma}

\begin{proof}

Let $w = w_j w_k^{-1} w_{k+1}$.
We compute the centroid of the alcove $G_{w^j} = w^{-j} \diamond G_\emptyset = G_\emptyset - (\underbrace{2, 2, \dots, 2}_j, 0, 0, \dots, 0)$ by Lemma \ref{lemma:www}. Therefore $w^j = z_\gamma$ by Lemma \ref{lemma:translates}.
\end{proof}

\begin{cor}
For $j \neq k$, $z_\gamma$ corresponds to the symmetric $2k$-core $\lambda = ((2k)^j, j^{2k-j})$. Equivalently, $z_\gamma$ corresponds to the shifted partition $(2k, 2k-1, \dots, 2k-j+1)$.
\end{cor}

\begin{proof}
Let $w = w_j w_k^{-1}w_{k+1}$. The first application of $w$ will add $2k-j+1$ boxes to the shifted diagram. Every subsequent application adds $2k-j+1$ boxes to a new row of the shifted diagram and one box to each previous row.
\end{proof}

The last case, when $\gamma = \Lambda_k^\vee$, is slightly different. We end this section by describing the corresponding symmetric $2k$-core in this case.

\begin{lemma}\label{lemma:z}
 If $\gamma  = \Lambda_k^\vee$ then $z_{\gamma} = w_k^{-1} w_{k-1}^{-1} \cdots w_1^{-1}$.
\end{lemma}

\begin{proof}
\[G_{w_k^{-1} w_{k-1}^{-1} \cdots w_1^{-1}} = (w_k^{-1} w_{k-1}^{-1} \cdots w_1^{-1})^{-1} \diamond G_\emptyset = w_1 \cdots w_{k-1} 
w_k \diamond G_\emptyset = \]  \[  (2-\frac{1}{k+1}, 2- \frac{2}{k+1}, \dots, 2-\frac{k}{k+1}) = (1,1,\dots, 1) + G_\emptyset = \gamma + G_\emptyset .\] By Lemma \ref{lemma:translates}, the statement follows.
\end{proof}
   
 \begin{lemma}\label{lemma:rectangle}
With the action on partitions described above, 
$$w_i^{-1} w_{i-1}^{-1} \cdots w_2^{-1} w_1^{-1} \emptyset = \underbrace{(i, i, \dots, i)}_{i}.$$ 
 \end{lemma}
 
 \begin{proof}
The proof is by induction. $w_1 = s_0$, and $s_0 \emptyset = (1)$. If $w_{i-1}^{-1}\cdots w_1^{-1} \emptyset = (i-1, i-1, \dots, i-1)$, then $w_i^{-1} (i-1, i-1, \cdots, i-1) = s_0 s_1 \cdots s_{i-1} s_i (i-1, \dots, i-1) = s_0 s_1 \cdots s_{i-1} (i, i-1, \dots, i-1, 1) = s_0 s_1 \cdots s_{i-2} (i, i, i-1, \dots, i-1, 2) = \dots = (i,i, \dots, i)$. 
 \end{proof}

\begin{cor}
$z_{\Lambda_k^\vee}$ corresponds to the symmetric $2k$-core \[\underbrace{(k, k, \dots, k)}_{k}.\] Equivalently, this corresponds to the shifted partition $(k, k-1, \dots, 2, 1)$. 
\end{cor} 

\begin{proof}
Follows from Lemma \ref{lemma:z} and Lemma \ref{lemma:rectangle}.
\end{proof}

\subsection{Subcores and a combinatorial formula}

We now illustrate our formulas for $k = 3$. We introduce the shorthand notation $\uu(i_1 i_2 \dots i_m)$ to denote $\uu(s_{i_1} s_{i_2} \cdots s_{i_m})$. The simplest example is $j=1$.

\begin{eg}
Let $j=1$. Then $z = z_{\Lambda_1^\vee} = s_1s_2s_3s_2s_1s_0$. The Dynkin automorphism $\tau$ corresponding to $z$ is trivial. 
 $w_0^1$ is the element  $s_1s_2s_3s_2s_1$.  Therefore $R = z \emptyset = (6, 1,1,1,1,1)$ and $S= w_0^1z\emptyset= (1)$. There are $6$ symmetric $6$-cores between $S$ and $R$, they are: 
\[ (1), (2,1), (3,1,1), (4,1,1,1), (5,1,1,1,1), (6,1,1,1,1,1) .\]

They correspond respectively to the following shifted diagrams.

\[\young{0&\bf\color{red}1&\bf\color{red}2&\bf\color{red}3&\bf\color{red}2&\bf\color{red}1\cr} \hspace{.1in}
\young{0&1&\bf\color{red}2&\bf\color{red}3&\bf\color{red}2&\bf\color{red}1\cr} \hspace{.1in}
\young{0&1&2&\bf\color{red}3&\bf\color{red}2&\bf\color{red}1\cr}
\]
\[\young{0&1&2&3&\bf\color{red}2&\bf\color{red}1\cr} \hspace{.1in}
\young{0&1&2&3&2&\bf\color{red}1\cr} \hspace{.1in}
\young{0&1&2&3&2&1\cr}
\]

Therefore \[\mathfrak{s}^C_{z_{\Lambda_1^\vee}} = \uu({0\bf\color{red}12321}) + \uu({10\bf\color{red}1232}) + \uu({210\bf\color{red}123}) \] \[+ \uu({3210\bf\color{red}12})+ \uu({23210\bf\color{red}1})+ \uu({123210}).\]

\end{eg}

\begin{eg}
Let $j=2$. Then $z = z_{\Lambda_2^\vee} = s_2s_3s_2s_1s_0s_2s_3s_2s_1s_0$. The Dynkin automorphism $\tau$ corresponding to $z$ is trivial. $w_0^2 = s_2s_1s_3s_2s_1s_3s_2$.  Therefore $R = z \emptyset = (6, 6,2,2,2,2)$ and $S= w_0^2z\emptyset= (2,2)$. There are $12$ symmetric $6$-cores between $S$ and $R$, they are: 
\[ (2,2), (3,2,1), (4,2,1,1), (3,3,2),\] \[(4,3,2,1), (5,2,1,1,1), (5,4,2,2,1),  (6,3,2,1,1,1),\] \[(6,4,2,2,1,1), (5,5,2,2,2), (6,5,2,2,2,1), (6,6,2,2,2,2) .\]

They correspond respectively to the following shifted diagrams.

\[\young{\blk&0&\bf\color{red}1&\bf\color{red}2&\bf\color{red}3&\bf\color{red}2\cr0&1&\bf\color{red}2&\bf\color{red}3&\bf\color{red}2&\bf\color{red}1\cr} \hspace{.1in}
\young{\blk&0&\bf\color{red}1&\bf\color{red}2&\bf\color{red}3&\bf\color{red}2\cr0&1&2&\bf\color{red}3&\bf\color{red}2&\bf\color{red}1\cr} \hspace{.1in}
\young{\blk&0&\bf\color{red}1&\bf\color{red}2&\bf\color{red}3&\bf\color{red}2\cr0&1&2&3&\bf\color{red}2&\bf\color{red}1\cr} \hspace{.1in}
\young{\blk&0&1&\bf\color{red}2&\bf\color{red}3&\bf\color{red}2\cr0&1&2&\bf\color{red}3&\bf\color{red}2&\bf\color{red}1\cr} \hspace{.1in}
\]
\[\young{\blk&0&1&\bf\color{red}2&\bf\color{red}3&\bf\color{red}2\cr0&1&2&3&\bf\color{red}2&\bf\color{red}1\cr} \hspace{.1in}
\young{\blk&0&\bf\color{red}1&\bf\color{red}2&\bf\color{red}3&\bf\color{red}2\cr0&1&2&3&2&\bf\color{red}1\cr} \hspace{.1in}
\young{\blk&0&1&2&\bf\color{red}3&\bf\color{red}2\cr0&1&2&3&2&\bf\color{red}1\cr} \hspace{.1in}
\young{\blk&0&1&\bf\color{red}2&\bf\color{red}3&\bf\color{red}2\cr0&1&2&3&2&1\cr} \hspace{.1in}
\]
\[\young{\blk&0&1&2&\bf\color{red}3&\bf\color{red}2\cr0&1&2&3&2&1\cr} \hspace{.1in}
\young{\blk&0&1&2&3&\bf\color{red}2\cr0&1&2&3&2&\bf\color{red}1\cr} \hspace{.1in}
\young{\blk&0&1&2&3&\bf\color{red}2\cr0&1&2&3&2&1\cr} \hspace{.1in}
\young{\blk&0&1&2&3&2\cr0&1&2&3&2&1\cr} \hspace{.1in}
\]

By Theorem \ref{thm:combinatorial}, \[\mathfrak{s}_{z_{\Lambda_2}^\vee}^C =  \uu({010\bf\color{red}2132132}) + \uu({0210\bf\color{red}232123}) + \uu({03210\bf\color{red}23212}) + \uu({10210\bf\color{red}23123})\]
 \[+ \uu({103210\bf\color{red}2312}) + \uu({023210\bf\color{red}2321})+ \uu({2103210\bf\color{red}231}) +\uu({1023210\bf\color{red}232})\]
  \[ +\uu({21023210\bf\color{red}23})+\uu({32103210\bf\color{red}21}) +\uu({321023210\bf\color{red}2}) +\uu({2321023210}).
\]

\end{eg}

\begin{eg}
Let $j= 3$. The word $z = z_{\Lambda_3^\vee}$ is $s_0s_1s_2s_0s_1s_0$. Then $z$ corresponds to the unique non-trivial Dynkin automorphism defined by $\tau(i) = 3-i$. The corresponding shifted diagram is $(3,2,1)$. Let $R = (3,3,3) = z\emptyset$ and $S = \tau(w_0^3) z\emptyset = \emptyset$ . There are 8 symmetric $6$ cores between $S$ and $R$. They are \[\emptyset, (1), (2,1), (2,2), (3,1,1), (3,2,1), (3,3,2), (3,3,3).\] These correspond respectively to the following shifted diagrams, where the bold letters correspond to elements not in $\lambda$ which have $\tau^{-1}$ applied to them.

\[\young{\blk&\blk&\bf\color{red}{3}\cr\blk&\bf\color{red}{3}&\bf\color{red}{2}\cr\bf\color{red}{3}&\bf\color{red}{2}&\bf\color{red}{1}\cr} \hspace{.1in}
\young{\blk&\blk&\bf\color{red}3\cr\blk&\bf\color{red}3&\bf\color{red}2\cr 0&\bf\color{red}2&\bf\color{red}1\cr} \hspace{.1in} 
\young{\blk&\blk&\bf\color{red}3\cr\blk&\bf\color{red}3&\bf\color{red}2\cr 0&1&\bf\color{red}1\cr} \hspace{.1in} 
\young{\blk&\blk&\bf\color{red}3\cr\blk&0&\bf\color{red}2\cr 0&1&\bf\color{red}1\cr} \hspace{.1in} \] \[
\young{\blk&\blk&\bf\color{red}3\cr\blk&\bf\color{red}3&\bf\color{red}2\cr 0&1&2\cr} \hspace{.1in} 
\young{\blk&\blk&\bf\color{red}3\cr\blk&0&\bf\color{red}2\cr 0&1&2\cr} \hspace{.1in} 
\young{\blk&\blk&\bf\color{red}3\cr\blk&0&1\cr 0&1&2\cr} \hspace{.1in} 
\young{\blk&\blk&0\cr\blk&0&1\cr 0&1&2\cr} \hspace{.1in} 
\]

By Theorem \ref{thm:combinatorial}, \[\mathfrak{s}_{z_{\Lambda_3}^\vee}^C =  \uu({\bf\color{red}321323})+\uu({0\bf\color{red}32312})+\uu({10\bf\color{red}3231})+\uu({010\bf\color{red}321})\] \[+\uu({210\bf\color{red}323})+\uu({0210\bf\color{red}32})+\uu({10210\bf\color{red}3})+\uu({010210}).\]
\end{eg}

\section{Remaining types}
\label{remaining types}

Although Hanusa and Jones did give descriptions of combinatorial affine Grassmannian sets for the type $B$ and $D$ cases, the combinatorics involved are not as nice. It seems plausible that some different collection of elements better suited to describing the terms appearing in expansions of $k$-Schur functions in these types will arise in the future. Rather than spending a good deal of space here to developing these in full generality, we will include the case of affine $B$ of rank $3$ and affine $D$ of rank $4$ as examples of what the combinatorics would look like; the compelled reader should easily be able to develop a corresponding expansion in full generality from these examples, the concepts of Section \ref{type c}, and a full understanding of Hanusa and Jones' combinatorics in these types.

\subsection{Affine type $B$, rank $3$}
Affine type $B$ has one non-trivial Dynkin diagram automorpism $\tau$, which is defined by permuting the indices $0$ and $1$, and fixing all other $i$.

\begin{eg}  The affine Grassmannian element $z = s_0 s_2 s_3 s_2 s_0$ corresponds to translation by the fundamental coweight $\Lambda^\vee_1$, which under the identification of Hanusa and Jones corresponds to the even symmetric $6$-core $(7,2,1,1,1,1,1)$:
\[ \young{0 \cr 0 \cr 2\cr 3\cr 2 \cr 0&0 \cr 0&0&2&3&2&0&0\cr }\]

This fundamental coweight corresponds to the nontrivial Dynkin automorphism $\tau$.
Again, as in type $C$, the objects involved in the bijection of Hanusa and Jones are symmetric cores, so we will remove half of the diagram, and study the skew partition: 

\[ \young{\blk&0 \cr 0&0&2&3&2&0&0\cr }\]

In this case, $\tau(w_0^1) = z$, so we need to look at all sub-diagrams between $S = \emptyset$ and $R = (7,1)$. There are six such diagrams:

\[\young{\blk&\bf\color{red}1 \cr \bf\color{red}1&\bf\color{red}1&\bf\color{red}2&\bf\color{red}3&\bf\color{red}2&\bf\color{red}1&\bf\color{red}1\cr } \hspace{.1in}
\young{\blk& 0 \cr 0&0&\bf\color{red}2&\bf\color{red}3&\bf\color{red}2&\bf\color{red}1&\bf\color{red}1\cr } \hspace{.1in}
\young{\blk&0 \cr 0 &0&2&\bf\color{red}3&\bf\color{red}2&\bf\color{red}1&\bf\color{red}1\cr } \] 
\[\young{\blk&0 \cr 0&0&2&3&\bf\color{red}2&\bf\color{red}1&\bf\color{red}1\cr } \hspace{.1in}
\young{\blk&0 \cr 0&0&2&3&2&\bf\color{red}1&\bf\color{red}1\cr } \hspace{.1in}
\young{\blk&0 \cr 0&0&2&3&2&0&0\cr } \]

By Theorem \ref{thm:combinatorial}, \[\mathfrak{s}_{z_{\Lambda_1}^\vee}^B = \uu({\bf\color{red}12321}) +
 \uu({0\bf\color{red}1232})  + \uu({20\bf\color{red}123}) +   \uu({320\bf\color{red}12}) + \uu({2320\bf\color{red}1}) + \uu(02320).\]

\end{eg}

\begin{eg}  The affine Grassmannian element $z = s_2 s_3 s_1 s_2 s_3s_1s_2s_0$ corresponds to translation by the Fundamental coweight $\Lambda^\vee_2$, which under the identification of Hanusa and Jones corresponds to the even symmetric $6$-core $(6,6,2,2,2,2)$:
\[ \young{1&2 \cr 2&3\cr 3&2\cr 2&1 \cr 0&0&1&2&3&2 \cr 0&0&2&3&2&1\cr }\]

This coweight corresponds to the trivial Dynkin automorphism, and the even symmetric $6$-core corresponds to the following skew partition: 

\[ \young{\blk&0&1&2&3&2 \cr 0&0&2&3&2&1\cr }\]

In this case, $w_0^2 = s_2s_3s_1s_2s_3s_1s_2$, so we need to look at all skew sub-diagrams between $S = w_0^2 z \emptyset =(2,1)$ and $R = (7,1)$. There are twelve such diagrams:

\[ 
\young{\blk&0&\bf\color{red}1&\bf\color{red}2&\bf\color{red}3&\bf\color{red}2 \cr 0&0&\bf\color{red}2&\bf\color{red}3&\bf\color{red}2&\bf\color{red}1\cr} \hspace{.1in}
\young{\blk&0&\bf\color{red}1&\bf\color{red}2&\bf\color{red}3&\bf\color{red}2 \cr 0&0&2&\bf\color{red}3&\bf\color{red}2&\bf\color{red}1\cr} \hspace{.1in}
\young{\blk&0&1&\bf\color{red}2&\bf\color{red}3&\bf\color{red}2 \cr 0&0&2&\bf\color{red}3&\bf\color{red}2&\bf\color{red}1\cr} \hspace{.1in}
\young{\blk&0&\bf\color{red}1&\bf\color{red}2&\bf\color{red}3&\bf\color{red}2 \cr 0&0&2&3&\bf\color{red}2&\bf\color{red}1\cr} \]
\[ \young{\blk&0&\bf\color{red}1&\bf\color{red}2&\bf\color{red}3&\bf\color{red}2 \cr 0&0&2&3&2&\bf\color{red}1\cr} \hspace{.1in}
\young{\blk&0&1&\bf\color{red}2&\bf\color{red}3&\bf\color{red}2 \cr 0&0&2&3&2&1\cr} \hspace{.1in}
\young{\blk&0&1&\bf\color{red}2&\bf\color{red}3&\bf\color{red}2 \cr 0&0&2&3&\bf\color{red}2&\bf\color{red}1\cr} \hspace{.1in}
\young{\blk&0&1&2&\bf\color{red}3&\bf\color{red}2 \cr 0&0&2&3&2&\bf\color{red}1\cr} \]
\[ \young{\blk&0&1&2&\bf\color{red}3&\bf\color{red}2 \cr 0&0&2&3&2&1\cr} \hspace{.1in}
\young{\blk&0&1&2&3&\bf\color{red}2 \cr 0&0&2&3&2&\bf\color{red}1\cr} \hspace{.1in}
\young{\blk&0&1&2&3&\bf\color{red}2 \cr 0&0&2&3&2&1\cr} \hspace{.1in}
\young{\blk&0&1&2&3&2 \cr 0&0&2&3&2&1\cr} \]

By Theorem \ref{thm:combinatorial}, \[\mathfrak{s}_{z_{\Lambda_2}^\vee}^B = \uu({0\bf\color{red}2132132}) + \uu({20\bf\color{red}213231}) + \uu({120\bf\color{red}21323}) +  \uu({320\bf\color{red}21321})\] 
\[ \uu({2320\bf\color{red}2321}) + \uu({12320\bf\color{red}232})   + \uu({3120\bf\color{red}2132}) + \uu({23120\bf\color{red}231})\] \[+ \uu({123120\bf\color{red}23})  + \uu({323120\bf\color{red}21})+
 \uu({1323120\bf\color{red}2}) + \uu({21323120}) .\]

\end{eg}

\begin{eg}  The affine Grassmannian element $z = s_3 s_2 s_3 s_0 s_2s_3s_1s_2s_0$ corresponds to translation by the fundamental coweight $\Lambda^\vee_3$, which under the identification of Hanusa and Jones corresponds to the even symmetric $6$-core $(7,6,6,4,3,3,1)$:
\[ \young{0 \cr 0&2&3\cr 2&3&2\cr 3&2&0&0 \cr 2&1&0&0&2&3 \cr 0&0&1&2&3&2\cr 0&0&2&3&2&0&0\cr }\]

This coweight corresponds to the nontrivial Dynkin automorphism $\tau$, and the even symmetric $6$-core corresponds to the following skew partition: 

\[ \young{\blk &\blk &\blk& 0\cr\blk &\blk &0&0&2&3 \cr \blk&0&1&2&3&2 \cr 0&0&2&3&2&0&0\cr }\]

In this case, $w_0^3 = s_3s_2s_1s_3s_2s_3$, so we need to look at all skew sub-diagrams between $S = \tau(w_0^3)z \emptyset = (3,2)$ and $R = z\emptyset = (7,5,4,1)$. There are eight such diagrams:

\[ \young{\blk &\blk &\blk& \bf\color{red}1\cr\blk &\blk &\bf\color{red}1&\bf\color{red}1&\bf\color{red}2&\bf\color{red}3 \cr \blk&0&1&\bf\color{red}2&\bf\color{red}3&\bf\color{red}2 \cr 0&0&2&\bf\color{red}3&\bf\color{red}2&\bf\color{red}1&\bf\color{red}1\cr } \hspace{.1in}
 \young{\blk &\blk &\blk& \bf\color{red}1\cr\blk &\blk &\bf\color{red}1&\bf\color{red}1&\bf\color{red}2&\bf\color{red}3 \cr \blk&0&1&\bf\color{red}2&\bf\color{red}3&\bf\color{red}2 \cr 0&0&2&3&\bf\color{red}2&\bf\color{red}1&\bf\color{red}1\cr } \hspace{.1in}
 \young{\blk &\blk &\blk& \bf\color{red}1\cr\blk &\blk &\bf\color{red}1&\bf\color{red}1&\bf\color{red}2&\bf\color{red}3 \cr \blk&0&1&2&\bf\color{red}3&\bf\color{red}2 \cr 0&0&2&3&2&\bf\color{red}1&\bf\color{red}1\cr } \hspace{.1in}
\young{\blk &\blk &\blk& 0\cr\blk &\blk &0&0&\bf\color{red}2&\bf\color{red}3 \cr \blk&0&1&2&\bf\color{red}3&\bf\color{red}2 \cr 0&0&2&3&2&0&0\cr }\]
\[ \young{\blk &\blk &\blk& \bf\color{red}1\cr\blk &\blk &\bf\color{red}1&\bf\color{red}1&\bf\color{red}2&\bf\color{red}3 \cr \blk&0&1&2&3&\bf\color{red}2 \cr 0&0&2&3&2&\bf\color{red}1&\bf\color{red}1\cr } \hspace{.1in}
 \young{\blk &\blk &\blk& 0\cr\blk &\blk &0&0&\bf\color{red}2&\bf\color{red}3 \cr \blk&0&1&2&3&\bf\color{red}2 \cr 0&0&2&3&2&0&0\cr } \hspace{.1in}
 \young{\blk &\blk &\blk& 0\cr\blk &\blk &0&0&2&\bf\color{red}3 \cr \blk&0&1&2&3&2 \cr 0&0&2&3&2&0&0\cr } \hspace{.1in}
\young{\blk &\blk &\blk& 0\cr\blk &\blk &0&0&2&3 \cr \blk&0&1&2&3&2 \cr 0&0&2&3&2&0&0\cr }\]

By Theorem \ref{thm:combinatorial}, \[\mathfrak{s}_{z_{\Lambda_3}^\vee}^B = \uu({120\bf\color{red}323123}) + \uu({3120\bf\color{red}32312}) +\uu({23120\bf\color{red}3231}) + \uu({023120\bf\color{red}323}) \] \[ \uu({323120\bf\color{red}321}) + \uu({3023120\bf\color{red}32}) + \uu({23023120\bf\color{red}3}) + \uu({323023120}).\]
\end{eg}

\subsection{Affine type $D$, rank $4$}

All Dynkin automorphisms in affine $D$ rank $4$ leave the index $2$ fixed. Here we give explicit expansions for two of the fundamental coweights.

 \begin{eg}  The affine Grassmannian element $z = s_0 s_2 s_4 s_1 s_2s_0$ corresponds to translation by the fundamental coweight $\Lambda^\vee_3$, which under the identification of Hanusa and Jones corresponds to the even symmetric $8$-core $(5,4,4,4,1)$:
\[ \young{4 \cr 4&2&0&0\cr 2&1&0&0\cr 0&0&1&2 \cr 0&0&2&4&4 \cr }\]

This coweight corresponds to a nontrivial Dynkin automorphism $\tau$ which swaps $0$ with $3$ and $1$ with $4$, and the even symmetric $8$-core corresponds to the following skew partition: 

\[ \young{\blk &\blk &\blk& 0\cr\blk &\blk &0&0 \cr \blk&0&1&2\cr 0&0&2&4&4\cr }\]

In this case, $w_0^3 = s_3s_2s_4s_1s_2s_3$, so we need to look at all skew sub-diagrams between $S = \tau(w_0^3)z \emptyset = \emptyset$ and $R = z\emptyset = (5,3,2,1)$. There are eight such diagrams:

\[
 \young{\blk &\blk &\blk& \bf\color{red}3\cr\blk &\blk &\bf\color{red}3&\bf\color{red}3 \cr \blk&\bf\color{red}3&\bf\color{red}4&\bf\color{red}2\cr \bf\color{red}3&\bf\color{red}3&\bf\color{red}2&\bf\color{red}1&\bf\color{red}1\cr } \hspace{.1in}
 \young{\blk &\blk &\blk& \bf\color{red}3\cr\blk &\blk &\bf\color{red}3&\bf\color{red}3 \cr \blk&0&\bf\color{red}4&\bf\color{red}2\cr 0&0&\bf\color{red}2&\bf\color{red}1&\bf\color{red}1\cr } \hspace{.1in}
 \young{\blk &\blk &\blk& \bf\color{red}3\cr\blk &\blk &\bf\color{red}3&\bf\color{red}3 \cr \blk&0&\bf\color{red}4&\bf\color{red}2\cr 0&0&2&\bf\color{red}1&\bf\color{red}1\cr } \hspace{.1in}
 \young{\blk &\blk &\blk& \bf\color{red}3\cr\blk &\blk &\bf\color{red}3&\bf\color{red}3 \cr \blk&0&\bf\color{red}4&\bf\color{red}2\cr 0&0&2&4&4\cr }
\]
\[
 \young{\blk &\blk &\blk& \bf\color{red}3\cr\blk &\blk &\bf\color{red}3&\bf\color{red}3 \cr \blk&0&1&\bf\color{red}2\cr 0&0&2&\bf\color{red}1&\bf\color{red}1\cr } \hspace{.1in}
 \young{\blk &\blk &\blk& \bf\color{red}3\cr\blk &\blk &\bf\color{red}3&\bf\color{red}3 \cr \blk&0&1&\bf\color{red}2\cr 0&0&2&4&4\cr } \hspace{.1in}
 \young{\blk &\blk &\blk& \bf\color{red}3\cr\blk &\blk &\bf\color{red}3&\bf\color{red}3 \cr \blk&0&1&2\cr 0&0&2&4&4\cr } \hspace{.1in}
 \young{\blk &\blk &\blk& 0\cr\blk &\blk &0&0 \cr \blk&0&1&2\cr 0&0&2&4&4\cr }
\]

By Theorem \ref{thm:combinatorial}, \[\mathfrak{s}_{z_{\Lambda_4}^\vee}^D = \uu({\bf\color{red}321423}) + \uu({0\bf\color{red}32142}) +\uu({20\bf\color{red}3241}) + \uu({420\bf\color{red}324}) \] \[ \uu({120\bf\color{red}321}) + \uu({4120\bf\color{red}32}) + \uu({24120\bf\color{red}3}) + \uu({024120}).\]
\end{eg}

 \begin{eg}  The affine Grassmannian element $z = s_0 s_2 s_3 s_1 s_2s_0$ corresponds to translation by the fundamental coweight $\Lambda^\vee_4$, which under the identification of Hanusa and Jones corresponds to the even symmetric $8$-core $(4,4,4,4)$:
\[ \young{3&2&0&0\cr 2&1&0&0\cr 0&0&1&2 \cr 0&0&2&3 \cr }\]

This coweight corresponds to a nontrivial Dynkin automorphism $\tau$ which swaps $0$ with $4$ and $1$ with $3$, and the even symmetric $8$-core corresponds to the following skew partition: 

\[ \young{\blk &\blk &\blk& 0\cr\blk &\blk &0&0 \cr \blk&0&1&2\cr 0&0&2&3\cr }\]

In this case, $w_0^4 = s_4s_2s_3s_1s_2s_4$, so we need to look at all skew sub-diagrams between $S = \tau(w_0^4)z \emptyset = \emptyset$ and $R = z\emptyset = (4,3,2,1)$. There are eight such diagrams:

\[
 \young{\blk &\blk &\blk& \bf\color{red}4\cr\blk &\blk &\bf\color{red}4&\bf\color{red}4 \cr \blk&\bf\color{red}4&\bf\color{red}3&\bf\color{red}2\cr \bf\color{red}4&\bf\color{red}4&\bf\color{red}2&\bf\color{red}1\cr } \hspace{.1in}
 \young{\blk &\blk &\blk& \bf\color{red}4\cr\blk &\blk &\bf\color{red}4&\bf\color{red}4 \cr \blk&0&\bf\color{red}3&\bf\color{red}2\cr 0&0&\bf\color{red}2&\bf\color{red}1\cr } \hspace{.1in}
 \young{\blk &\blk &\blk& \bf\color{red}4\cr\blk &\blk &\bf\color{red}4&\bf\color{red}4 \cr \blk&0&\bf\color{red}3&\bf\color{red}2\cr 0&0&2&\bf\color{red}1\cr } \hspace{.1in}
 \young{\blk &\blk &\blk& \bf\color{red}4\cr\blk &\blk &\bf\color{red}4&\bf\color{red}4 \cr \blk&0&\bf\color{red}3&\bf\color{red}2\cr 0&0&2&3\cr }
\]
\[
 \young{\blk &\blk &\blk& \bf\color{red}4\cr\blk &\blk &\bf\color{red}4&\bf\color{red}4 \cr \blk&0&1&\bf\color{red}2\cr 0&0&2&\bf\color{red}1\cr } \hspace{.1in}
 \young{\blk &\blk &\blk& \bf\color{red}4\cr\blk &\blk &\bf\color{red}4&\bf\color{red}4 \cr \blk&0&1&\bf\color{red}2\cr 0&0&2&3\cr } \hspace{.1in}
 \young{\blk &\blk &\blk& \bf\color{red}4\cr\blk &\blk &\bf\color{red}4&\bf\color{red}4 \cr \blk&0&1&2\cr 0&0&2&3\cr } \hspace{.1in}
 \young{\blk &\blk &\blk& 0\cr\blk &\blk &0&0 \cr \blk&0&1&2\cr 0&0&2&3\cr }
\]

By Theorem \ref{thm:combinatorial}, \[\mathfrak{s}_{z_{\Lambda_4}^\vee}^D = \uu({\bf\color{red}421324}) + \uu({0\bf\color{red}42132}) +\uu({20\bf\color{red}4231}) + \uu({320\bf\color{red}423}) \] \[ \uu({120\bf\color{red}421}) + \uu({3120\bf\color{red}42}) + \uu({23120\bf\color{red}4}) + \uu({023120}).\]
\end{eg}

\bibliographystyle{amsalpha}

\end{document}